\pdfoutput=1
\documentclass[11pt]{amsart}

\usepackage[utf8]{inputenc}
\usepackage[T1]{fontenc}
\usepackage{lmodern} 
\usepackage[english]{babel} 

\usepackage{amsmath,amsfonts,amsthm,amssymb,mathtools,mathdots,mathrsfs,stmaryrd,bbold}

\usepackage{graphicx,xcolor}
\usepackage{tikz}
\usepackage{tikz-cd}
\usepackage{pgfplots}
\pgfplotsset{compat=1.15}

\usetikzlibrary{cd,calc,decorations.pathmorphing,shapes.multipart}

\usepackage[margin=3cm]{geometry}
\usepackage{marginnote,url,caption,subcaption,fancyhdr,fancybox,enumitem,csquotes,setspace,xstring,ifthen,import}

\newtheoremstyle{mytheoremstyle}{6pt}{6pt}{\slshape}{}{\bfseries}{.}{1em}{}
\theoremstyle{mytheoremstyle}
\newtheorem{thm}{Theorem}[section]
\newtheorem{lem}[thm]{Lemma}

\newtheorem{prop}[thm]{Proposition}
\newtheorem{ddef}[thm]{Definition}
\newtheorem{rmq}[thm]{Remark}

\DeclareMathOperator{\minimum}{min}
\DeclareMathOperator{\degree}{deg}
\DeclareMathOperator{\gonality}{gon}
\DeclareMathOperator{\covgon}{cov.gon}
\DeclareMathOperator{\order}{ord}
\DeclareMathOperator{\dimension}{dim}
\DeclareMathOperator{\irr}{irr}
\DeclareMathOperator{\supp}{supp}

\makeatletter
  \renewcommand\@seccntformat[1]{\csname the#1\endcsname.\quad}
\makeatother

\begin{document}

\title{Covering gonality of hypersurfaces in a product of projective spaces}
\author{Raphaël Hiault}

\date{}
\maketitle
\begin{abstract}
In this work, we investigate the behaviour of the covering gonality of a very general hypersurface in a product of projective spaces. Inspired by the work of Bastianelli, Ciliberto, Flamini and Suppino in  \cite{bastianelli2019gonality} which addresses the case of a hypersurface in a projective space, we establish a similar result for very general smooth hypersurfaces of sufficiently large bi-degree. More precisely, we show that the covering gonality of such hypersurfaces can be computed by viewing them as a family of hypersurfaces over a projective space. Then, the curves computing the covering gonality lie entirely within the fibres of one of the families. This rules out transversal curves from computing the covering gonality. In addition to this, we investigate the behaviour of the joint covering gonality as in \cite{lazarsfeld2023measures} and establish a lower bound for bi-degree large enough. 
\end{abstract}
\begin{center}

\section*{Introduction}

\end{center}

\mdseries
A central problem in algebraic geometry is to determine the extent to which a given algebraic variety $X$ admits rationality properties—such as rationality, unirationality, or rational connectedness. Within the past few years, natural invariants which measure these rationality properties have been extensively studied. The foundational case is the one-dimensional setting. Consider $C$, an irreducible smooth curve. The gonality of $C$ is defined as the minimal degree of a morphism  $\phi : C \rightarrow \mathbb{P}^{1}$ and we denote it $\gonality(C)$. In the higher dimensional setting, various extensions of this definition exist, depending on the rationality property of interest. Consider a smooth projective variety $X$ of dimension $n$. On the one hand, the degree of irrationality of $X$ is defined as the minimal degree of a rational dominant map $f : X \dashrightarrow \mathbb{P}^{n}$. We denote it $\irr(X)$.  Then, $X$ is rational if and only if $\irr(X)=1$. This definition coincides with the gonality in the one-dimensional case. On the other hand, the covering gonality of $X$ is defined as the minimal gonality of a curve $C$ passing through a general point $x \in X $. We denote it $\covgon(X)$. One can notice that $X$ is uniruled if and only if $\covgon(X)=1$. \\
Historically, the first result is due to Max Noether who showed that the gonality of a smooth plane curve $C$ of degree $d\geq 3$ is computed by the projection from a point $x\in C$. This result was subsequently generalized in \cite{degreehypersurfaces} for a very general irreducible smooth hypersurface $X\subset\mathbb{P}^{n}$ of degree $d\geq 2n+1$. The main result of \cite{degreehypersurfaces} is that $\irr(X)$ is computed by the projection from a point $x\in X$.  More generally, these invariants are actively studied in different settings, see e.g.,  \cite{chen2020fano}, \cite{stapleton2020degree}, \cite{chen2023rational}, \cite{chen2021multiplicative}, \cite{colombo2022degree}, \cite{bastianelli2022covering}, \cite{moretti2024degree}, \cite{voisin2022fibrations} and for a global survey the reader can consult \cite{chen2025primermeasuresirrationality}. In this note, we are interested in the covering gonality. In recent years, several results have been established in various contexts. For instance, in \cite{martin2020conjecture},  Martin exhibits a lower bound for the covering gonality of a very general abelian variety in any dimension answering a conjecture of Voisin in \cite{voisin2018chow}. Also, in \cite{chen2021multiplicative}, the author gives a multiplicative lower bound for the covering gonality of certain complete intersections. Finally, in \cite{bastianelli2019gonality} the authors compute the covering gonality of a very general irreducible smooth hypersurface $X\subset \mathbb{P}^{n+1}$ of degree $d\geq 2n+2$. This work actually originated in \cite{lopez1995curves}, \cite{degreehypersurfaces} in which the authors investigated the lower dimensional cases.

Inspired by the result \cite[Theorem 1.1]{bastianelli2019gonality}, we prove the following theorem:

\vspace{0.35cm}

\begin{thm}
\label{mainthm}
Let $X \subset \mathbb{P}^{n+1}\times \mathbb{P}^{m+1}$ be a very general irreducible smooth hypersurface of bi-degree $(a,b)$. Suppose that $a\geq 2n+m+1$, $b\geq 2m+n+1$, and $n,m \geq 2 $. Then:
\begin{center}
    ${ \minimum \Big\{a-\lfloor \frac{\sqrt{16n+9}-1}{2} \rfloor,b-\lfloor \frac{\sqrt{16m+9}-1}{2} \rfloor \Big\} \leq \covgon(X)\leq \minimum \Big\{a-\lfloor \frac{\sqrt{16n+1}-1}{2} \rfloor,b-\lfloor \frac{\sqrt{16m+1}-1}{2} \rfloor \Big\} }$.
\end{center}
Moreover, both sides of the inequality coincide whenever $m,n \in \{4\alpha^{2}+3\alpha, 4\alpha^{2}+5\alpha+1, \alpha \in \mathbb{N}\}$
\end{thm}
\vspace{0.35cm}
First, this computes the covering gonality of such hypersurfaces up to one, in the same spirit as the result of \cite{bastianelli2019gonality}. This result admits a geometric interpretation, which lies at the core of this work. Consider $X \subset \mathbb{P}^{n+1} \times \mathbb{P}^{m+1}$ a very general smooth irreducible of bi-degree $(a,b)$ with $a \geq 2n+m+1, b \geq 2m+n+1$. Via the natural projection onto each factor, $X$ may alternatively be seen as a family of degree $a$ hypersurfaces in $\mathbb{P}^{n+1}$ parametrized by $\mathbb{P}^{m+1}$, or as a family of hypersurfaces of degree $b$ in $\mathbb{P}^{m+1}$ parametrized by $\mathbb{P}^{n+1}$. We prove that the covering gonality of $X$ is actually the covering gonality of a general fibre from one of these families. \\
In the final section of this paper we study the newly introduced birational invariant measure of association. This notion has been introduced by Lazarsfeld-Martin in \cite{lazarsfeld2023measures} and the intuition is as follow: given a pair of two algebraic variety of the same dimension $n$, one can be interested by the following question: are $X$ and $Y$ birationally related? If not how can one measure that failure? In \cite{lazarsfeld2023measures} Lazarfeld and Martin introduce the joint covering gonality of two algebraic varieties of same dimension $X$ and $Y$, that we denote $\covgon(X,Y)$. Although less understood than in the classical setting, this invariant already satisfies interesting lower bounds. Our second main result establishes such a bound for hypersurfaces in products of projective spaces inspired by the one proven in \cite{lazarsfeld2023measures}.

\begin{thm}
Let $X, Y \subset \mathbb{P}^{n+1} \times \mathbb{P}^{m+1}$ be very general smooth hypersurfaces 
of bi-degrees $(a_1,b_1)$ and $(a_2,b_2)$ respectively with $a_{1},a_{2} \geq 2n+m+2$ and $b_{1},b_{2} \geq 2m+n+2$. Then:
\vspace{0.5cm}
\begin{center}
   $\operatorname{cov.gon}(X,Y) \;\geq\; k_1 + k_2 - n - m$
\end{center}
where : 
\[
k_1 = \min\{a_1 - n - 3, b_1 - m - 3\}, \quad
k_2 = \min\{a_2 - n - 3, b_2 - m - 3\}.
\]

\end{thm}
 
Concerning the organisation of this paper it will be as follows. In the first section we recall all the needed background on the covering gonality and the classical way one has to control this invariant. The second section, much shorter is devoted to define the notion of line in a product of projective spaces which we use heavily during the proof of the first theorem in Section 3. Finally, in section 4 we explore the joint covering gonality and prove the second main theorem stated above.
\\

 In this work, $X$ will denote a smooth irreducible  complex projective variety of  dimension $\dim(X)$. We will say a property holds at a general point $x \in X $ if it holds on a Zariski open subset of $X$. We will say that a property holds for a very general point in $X$ if it holds outside of countably many proper Zariski closed subsets in $X$. 

\vspace{0.3cm}

\textit{Acknowledgments}: 
I would first like to express my sincere gratitude to my advisors, Damian Brotbek and Gianluca Pacienza, for their time, guidance, and valuable feedback on this work and its earlier versions. This paper owes a great deal to Francesco Bastianelli, who delivered a series of lectures at the IECL during which he generously proposed this problem. The present work is deeply inspired by his research. Finally, I warmly thank Steve Balme for his encouragement throughout the development of this work.
 
\begin{center}
\section{Classical results around the covering gonality}
\end{center}
In this first section, we review fundamental results that will be used throughout the paper. 
\vspace{0.2cm}

We start by recalling what a covering family of $c$-gonal curves is: 
\vspace{0.2cm}
\begin{ddef}
 A family of $c$-gonal curves over a variety $X$ consists of a smooth morphism $\eta : \mathcal{C} \longrightarrow \mathcal{T} $ whose generic fibre $\eta^{-1}(t), t \in \mathcal{T}$, is a smooth irreducible curve of gonality $c$, together with a dominant rational morphism $f : \mathcal{C} \longrightarrow X$,  such that the restriction map  $f\vert_{\mathcal{C}_{t}} : \mathcal{C}_{t} \longrightarrow X $ is birational onto its image for general $t \in \mathcal{T}$.
\end{ddef}

\begin{rmq}
Following \cite[Remark 1.5]{degreehypersurfaces}, we may assume that $T$ and $\mathcal{C}$ are smooth, and that $\dimension(T)=\dimension(X)-1$, $\dimension(\mathcal{C})=\dimension(X)$. 
\end{rmq}

Using the notion of a covering family of curves, we can now define the covering gonality: 

\begin{ddef}
 The covering gonality of $X$, denoted $\covgon(X)$, is the minimal integer $c$ such that $X$ admits a covering family of $c$-gonal curves. 
\end{ddef}

We now fix a covering family of $c$-gonal curves over $X$ i.e we have a covering family of curves $\eta$ $:$ $\mathcal{C} \longrightarrow T$, with the dominant map $f : \mathcal{C} \longrightarrow X$.  After a suitable  base change, one can assemble the morphism computing the gonality of $\mathcal{C}_{t}$. More precisely, the following diagram commutes: \\
\begin{center}
\begin{tikzcd}
	{\mathcal{C}} &&& {T\times \mathbb{P}^{1}} \\
	\\
	&&& T
	\arrow["\phi", from=1-1, to=1-4]
	\arrow["\eta"', from=1-1, to=3-4]
	\arrow["{pr_{1}}", from=1-4, to=3-4]
\end{tikzcd}
\end{center}

\vspace{0.35cm}

For a general $t \in T$, the morphism $\phi\vert_{\mathcal{C}_{t}} : \mathcal{C}_{t} \longrightarrow \mathbb{P}^{1}$ is a morphism computing the gonality of the curve $\mathcal{C}_{t}$. Since each $\mathcal{C}_{t}$ has gonality $c$, the restriction $\phi\vert_{\mathcal{C}_{t}}$ is a degree $c$ morphism and a general fibre consists of $c$ points. We denote by  $\phi^{-1}\vert_{\mathcal{C}_{t}}(y)=\phi^{-1}(t,y)=\{q_{1},...,q_{c}\} \subset \mathcal{C}_{t}$ the fibre of the map $\phi\vert_{\mathcal{C}_{t}}$ over general points $y \in \mathbb{P}^{1}$ and $t \in T $. From the map $\phi$, the author in \cite{bastianelli2012symmetric} constructs a correspondence of degree $c$ in $X \times (T \times  \mathbb{P}^{1})$ i.e. they build a variety $\Gamma \subset X \times (T\times \mathbb{P}^{1})$ which is the graph of $\phi$ together with the projection $\pi : \Gamma  \longrightarrow T \times \mathbb{P}^{1}$ which is a generically finite morphism of degree $c$. 
A more detailed version of this construction can be found in \cite[Example 4.7]{bastianelli2012symmetric}, particularly the base change part.  \\
Exploiting this, the work in \cite{gonalitynoether} describes some geometrical restrictions on the points $f(q_{1}),...,f(q_{c})$: 
\begin{ddef}
Let $\mathcal{L}$ be a linear system on $X$ and $Z=\{P_{1},...,P_{k}\} \subset X$ a set of $k$ points. We say that $Z$ satisfies the Cayley-Bacharach condition with respect to $\mathcal{L}$ if for every effective divisor $D \in \mathcal{L}$ such that $P_{1},...,P_{i-1},P_{i+1},...,P_{k} \in \supp(D)$, then $P_{i} \in D$.
\end{ddef}

The two main results we will use are the two following statements from  \cite{gonalitynoether}:

\begin{prop}[{\cite[Proposition 2.3]{gonalitynoether}}]
\label{Cayley}
Let $\Gamma$ be the correspondence constructed above. Then for a general $t \in T$, the fibre of the map  $\phi\vert_{\mathcal{C}_{t}}$ satisfies the Cayley-Bacharach condition with respect to $\vert K_{X} \vert$.
\end{prop}

\begin{lem}[{\cite[Lemma 2.4]{gonalitynoether}}]
\label{CayleyLine}
Let $Z=\{P_{1},...,P_{k}\} \subset \mathbb{P}^{n}$ be a set of $k$ points satisfying the Cayley-Bacharach condition with respect to the linear system $\vert \mathscr{O}_{\mathbb{P}^{n}}(m)\vert$. Then $k\geq m+2$. Moreover, if $k\leq 2m+1$, then the points in $Z$ lie on a line in $\mathbb{P}^{n}$.
\end{lem}

These two results are central to our proof as the canonical bundle of a hypersurface in $\mathbb{P}^{n+1}\times \mathbb{P}^{m+1}$ can be computed explicitly. In addition, Proposition 1.5 highlights the role of the canonical bundle's positivity: the more positive the canonical bundle is, the more geometric constraints it imposes on the fibre of the correspondence.
To conclude this section, we recall an inequality from \cite[Lemma 6.7]{stapleton2020degree}. This result was initially established for a very general smooth variety in a projective space in \cite[Proposition 3.8]{degreehypersurfaces} using a variational method developed by Voisin and Ein developed in \cite{voisin1996conjecture}, \cite{ein1988subvarieties}. In \cite{stapleton2020degree}, the authors later generalized this inequality to the setting of very general smooth hypersurfaces in products of projective spaces. 
\begin{lem}[{\cite[Lemma 6.7]{stapleton2020degree}}]
\label{StapUll}
Consider $X \subset \mathbb{P}^{n+1} \times \mathbb{P}^{m+1}$ a very general, smooth hypersurface of bi-degree $(a,b)$ with $a,b\geq 2$ . Then, for any subvariety $E\subset X$ swept out by curves of gonality $c$, the following inequality holds: 
\begin{center}
    $c \geq \dimension(E) + \minimum\{a-n-1,b-m-1\}-n-m$.
\end{center}
\end{lem}

\begin{rmq}
\label{NoLine}
For $c=1$ and $a \geq 2n+m+1, b\geq 2m+n+1$, then X contains no rational curves. We will use later on this remark. 
\end{rmq}

\begin{center}
\section{Lines over $\mathbb{P}^{n+1}\times \mathbb{P}^{m+1}$}
\end{center}
 In this section we analyze lines in $\mathbb{P}^{n+1}\times\mathbb{P}^{m+1}$. Lemma 1.6 indicates that the fibres of a correspondence $Z$ will lie on lines as long as they satisfy the Cayley-Bacharach condition with respect to a multiple of  $\mathscr{O}_{\mathbb{P}^{n+1}}(1)$. It is natural, then, to ask how one can generalize this in $\mathbb{P}^{n+1} \times \mathbb{P}^{m+1}$. We say that $\ell$ is a line in  $\mathbb{P}^{n+1}\times \mathbb{P}^{m+1}$ if it is an irreducible curve such that $\degree_{\mathscr{O}_{\mathbb{P}^{n+1}\times \mathbb{P}^{m+1}}(1,1)}(\ell)=1$ where $\degree_{\mathscr{O}_{\mathbb{P}^{n+1}\times \mathbb{P}^{m+1}}(1,1)}(\ell)$ is  the degree of $\ell$  with respect to $\mathscr{O}_{\mathbb{P}^{n+1}\times \mathbb{P}^{m+1}}(1,1)$. To lighten the notation, whenever no confusion is possible, the line bundle $\mathscr{O}_{\mathbb{P}^{n+1}\times \mathbb{P}^{m+1}}(1,1)$ will be denoted $\mathscr{O}(1,1)$. We denote by $F(\mathbb{P}^{n+1} \times \mathbb{P}^{m+1})$ the Fano variety of lines of $\mathbb{P}^{n+1}\times \mathbb{P}^{m+1}$. The next proposition shows that a line in $\mathbb{P}^{n+1} \times \mathbb{P}^{m+1}$ is actually either a line in $\mathbb{P}^{n+1}$ or in  $\mathbb{P}^{m+1}$.
\vspace{0.35cm}

\begin{prop}
\label{FanoOfLine}
 For any  $m,n \in \mathbb{N}$ the Fano variety of lines of $\mathbb{P}^{n+1} \times \mathbb{P}^{m+1}$ decomposes as:
\begin{center}
$ F(\mathbb{P}^{n+1} \times \mathbb{P}^{m+1})=(Gr(1,n+1)\times\mathbb{P}^{m+1}) \sqcup (\mathbb{P}^{n+1} \times Gr(1,m+1) )$
\end{center}

where $Gr(1,n+1)$ is the Grassmannian of lines of $\mathbb{P}^{n+1}$ and $Gr(1,m+1)$ is the Grassmannian of lines of $\mathbb{P}^{m+1}$.

\end{prop}

\begin{proof}
Let $\ell$ be a line in $\mathbb{P}^{n+1}\times \mathbb{P}^{m+1}$ and  $i : \ell \hookrightarrow \mathbb{P}^{n+1}\times \mathbb{P}^{m+1}$ the embedding of $\ell$ in $\mathbb{P}^{n+1}\times \mathbb{P}^{m+1}$. The immersion $i$ comes with two natural morphisms, $i_{1}=\pi_{1}\circ i$, $i_{2}=\pi_{2} \circ i$ that are the projections of $\ell$ respectively in $\mathbb{P}^{n+1}$ and $\mathbb{P}^{m+1}$. Suppose the images of $i_{1}$  and $i_{2}$ are one-dimensional. Then $\degree_{\mathscr{O}(1,1)}(\ell)=\degree(i_{1}(\ell))+\degree(i_{2}(\ell))$, where the degree of $i_{1}(\ell)$ and $i_{2}(\ell)$ are computed with respect to the hyperplane divisor. But as  $\degree_{\mathscr{O}(1,1)}(\ell)=1$, this implies that one of the curves $i_{1}(\ell),i_{2}(\ell)$ is of degree 0, which is a contradiction. This means one of the map $i_{1},i_{2}$ is constant, while the other is a line in $\mathbb{P}^{n+1}$ or $\mathbb{P}^{m+1}$. 
\end{proof}

\vspace{0.35cm}

\begin{center}
\section{Lower bound for the covering gonality}
\end{center}
In this section, adapting techniques from \cite[Section 2.4]{bastianelli2019gonality}, we derive a lower bound for $\covgon(X)$.  The approach of \cite{bastianelli2019gonality} mostly consists in showing that a curve $C$ passing through a general point $x$ of general hypersurface $X \subset \mathbb{P}^{n}$ of degree $d \geq 2n+2$ and computing $\covgon(X)$ must live in the cone of lines passing through $x$ with some multiplicity (depending on the gonality of the curve and the degree of the hypersurface). Here, we show that the same holds with our previously defined lines for a hypersurface in a product of projective space. Then, we will rely on the work of Stapleton and Ullery in  \cite{stapleton2020degree} mentioned above. The obtained geometrical constraints on the curves together with the idea of seeing $X$ as a family of hypersurfaces enable us to show that these curves must lie in one of the fibres of the family of hypersurfaces. In particular, this means that the transversal curves i.e. curves that do not lie in one of the fibres of the family, do not compute $\covgon(X)$. \\
Now let us fix the notation. Let $X$ be a very general smooth hypersurface in $\mathbb{P}^{n+1} \times \mathbb{P}^{m+1}$ of bi-degree $(a,b)$ with $n,m \in \mathbb{N}^{*}$. In our case, a standard computation yields: \\
 \begin{center}
 $\omega_{X}=\mathscr{O}_{\mathbb{P}^{n+1} \times
 \mathbb{P}^{m+1}}(a-n-2,b-m-2)$\\
 \end{center}
 and we denote 
\[
    k:=min\{a-n-2,b-m-2\}
\]
We can remark that to have $\omega_{X}$ ample is equivalent to have $a \geq n+3, b\geq m+3$, i.e. $k\geq 1$.
First, inspired by \cite[Theorem 3.5]{gonalitynoether}, we show that for a curve $C$ lying on $X$, the fibre of a morphism computing the gonality of $C$ must be on a line :

\begin{prop}
\label{FibreCorres}
Let $X$ be a hypersurface in $\mathbb{P}^{n+1} \times \mathbb{P}^{m+1}$ of bi-degree $(a,b)$ and assume that $k=min\{(a-n-2,b-m-2)\} \geq 1$. Suppose we have a covering family of curves of gonality $c$ and let $\Gamma \subset X\times (T \times \mathbb{P}^{1})$ be its associated correspondence of degree $c$ with the projection $\pi : \Gamma \longrightarrow T\times \mathbb{P}^{1} $.  If $c\leq 2k+1$,  then for a general $y\in \mathbb{P}^{1}$ and $t\in T$, the points in $\phi^{-1}(t,y)=\{f(q_{1}),...,f(q_{c})\}$, where $q_{1},...,q_{c} \in \mathcal{C}_{t}$, lie on a line in $\mathbb{P}^{n+1} \times \mathbb{P}^{m+1} $.
\end{prop}

\begin{proof}
 Consider the Segre embedding of $X$, i.e., the map $i : X \hookrightarrow \mathbb{P}^{(n+2)(m+2)-1}$ which embeds $X$ as subvariety of a projective space. Let $t\in T,y \in \mathbb{P}^{1}$ be general points, then according to Lemma \ref{Cayley} the points $\phi^{-1}((y,t))=\{f(q_{1}),...,f(q_{c})\}$ satisfy the Cayley-Bacharach condition with respect to $\vert K_{X} \vert$. As, $i^{*} \mathscr{O}(1) \cong \mathscr{O}(1,1)$, we get that $\omega_{X}=\mathscr{O}(a-n-2,b-m-2)=\mathscr{O}(k,k)\otimes \mathscr{O}(a-n-2-k,b-n-2-k) \cong \mathscr{O}(k,k) \otimes \mathscr{O}(E) \cong i^{*}\mathscr{O}(k) \otimes \mathscr{O}(E)$, where $E$ is an effective divisor or a zero divisor. Note that $E$ may be selected so that it avoids the points in $\phi^{-1}((y,t))=\{f(q_{1}),...,f(q_{c})\}$. Now, as $c\leq 2k+1$, Lemma \ref{CayleyLine} yields that the family of $c$-points $i(\phi^{-1}((y,t))=\{i(f(q_{1})),...,i(f(q_{c}))\}$ are on a line $\ell$ in $\mathbb{P}^{(n+2)(m+2)-1}$. This line must be contained in the image of the Segre embedding. Indeed, the image of the Segre embedding is cut out by equations of degree two in $\mathbb{P}^{(n+2)(m+2)-1}$. But $\ell$ intersects $X$ at least at $c$ points which are the points in $i(\phi^{-1}((y,t))$. Moreover, Lemma \ref{CayleyLine} ensures that $c\geq k+2\geq 3$. Then, by Bézout's theorem, the line $\ell$ must lie in the image of the Segre embedding. As $i^{*}\mathscr{O}(1) \cong \mathscr{O}(1,1)$, we conclude that the pullback of the line $\ell$ is a line in $\mathbb{P}^{n+1}\times \mathbb{P}^{m+1}$. This means that the points $f(q_{1}),...,f(q_{c})$ lie on a line in $\mathbb{P}^{n}\times \mathbb{P}^{m}$. 

\end{proof}
From now on, we assume that $a \geq 2n+m+1$ and $b \geq 2m+n+1$. Let $\eta : \mathcal{C} \longrightarrow T, \ f : \mathcal{C} \longrightarrow X$ be a covering family of $c$-gonal curves on $X$, with $c \leq 2k+1$. For general points $y \in \mathbb{P}^{1}$ and $t \in T$, Proposition \ref{FibreCorres} shows that the points in $\phi^{-1}(t,y)=\{f(q_{1}),\dots,f(q_{c})\}$ lie on a line, which we denote by $\ell_{(t,y)} \subset \mathbb{P}^{n+1}\times\mathbb{P}^{m+1}$.

Since $X$ is a very general hypersurface with $a \geq 2n+m+1$ and $b \geq 2m+n+1$, Lemma \ref{NoLine} implies that $X$ does not contain any rational curve. In particular, the schematic intersection $\ell_{(t,y)} \cdot X$ is a zero--cycle on $X$. As $(t,y)\in T\times\mathbb{P}^{1}$ varies, this construction defines a family of lines whose image has dimension $m+n+1$ inside the Fano variety $F(\mathbb{P}^{n+1}\times\mathbb{P}^{m+1})$. We denote by $\mathcal{B} := \overline{\{\ell_{(t,y)} \mid (t,y)\in T\times\mathbb{P}^{1}\}} \subset F(\mathbb{P}^{n+1}\times\mathbb{P}^{m+1})$ the associated congruence of lines, and we fix once and for all a desingularization $\widetilde{\mathcal{B}} \longrightarrow \mathcal{B}$.

We consider the corresponding incidence variety 
$$\mathcal{I} := \{(\ell,x)\mid \ell\in \widetilde{\mathcal{B}},\ x\in \ell\} \subset \widetilde{\mathcal{B}}\times (\mathbb{P}^{n+1}\times\mathbb{P}^{m+1})$$,
together with the natural projections $p : \mathcal{I} \longrightarrow \widetilde{\mathcal{B}}, \ \mu : \mathcal{I} \longrightarrow \mathbb{P}^{n+1}\times\mathbb{P}^{m+1}$. This yields the following diagram:
\begin{center}
\begin{tikzcd}[column sep=4em, row sep=2em, ampersand replacement=\&]
	\mathcal{I} \arrow[r, "p"] \arrow[rr, bend left=20, "\mu"] \& \widetilde{\mathcal{B}} \arrow[d] \& \mathbb{P}^{n+1}\times\mathbb{P}^{m+1} \\
	\& \mathcal{B} \& 
\end{tikzcd}
\end{center}
Since $T\times\mathbb{P}^{1}$ is irreducible, the congruence $\mathcal{B}$ is contained in a single irreducible component of $F(\mathbb{P}^{n+1}\times\mathbb{P}^{m+1}) = Gr(1,n+1)\times \mathbb{P}^{m+1} \;\sqcup\; \mathbb{P}^{n+1}\times Gr(1,m+1)$. Up to exchanging the two factors, we may therefore assume that $\mathcal{B} \subset Gr(1,n+1)\times \mathbb{P}^{m+1}$.

The above construction admits a fiberwise version. More precisely, let $z\in\mathbb{P}^{m+1}$ and denote by $\mathcal{B}_{z} := \{\ell\in\mathcal{B}\mid \pi_{2}(\ell)=z\} \subset Gr(1,n+1)$ the induced congruence of lines over $z$. The fiber $X_{z}:=\pi_{2}^{-1}(z)$ is a smooth hypersurface of degree $a$ in $\mathbb{P}^{n+1}$. Let $\widetilde{\mathcal{B}}_{z}$ be the inverse image of $\mathcal{B}_{z}$ in $\widetilde{\mathcal{B}}$, and define the corresponding incidence variety $\mathcal{I}_{z} := p^{-1}(\widetilde{\mathcal{B}}_{z})$. The restriction of $\mu$ induces a morphism $\mu_{z} : \mathcal{I}_{z} \longrightarrow \mathbb{P}^{n+1}$, summarized in the following diagram:
\begin{center}
\begin{tikzcd}[column sep=4em, row sep=2em, ampersand replacement=\&]
	\mathcal{I}_{z} \arrow[r, "p"] \arrow[rr, bend left=20, "\mu_{z}"] \& \widetilde{\mathcal{B}}_{z} \arrow[d] \& \mathbb{P}^{n+1} \\
	\& \mathcal{B}_{z} \& 
\end{tikzcd}
\end{center}
This setting allows a finer control on the gonality of the curves in the covering family, depending on the geometry of the congruence $\mathcal{B}_{z}$. This will be crucial in the remainder of the proof.

\begin{prop}
\label{Estimate}
Let $a\geq 2n+m+1,b\geq 2m+n+1$ and $B$ the congruence of lines defined of the covering family of curves of gonality $c$ as defined above. First, either $\mathcal{B} \subset Gr(1,n+1) \times \mathbb{P}^{m+1}$ or $B\subset \mathbb{P}^{n+1} \times Gr(1,m+1)$ and if $B \subset Gr(1,n+1) \times \mathbb{P}^{m+1}$ (resp. $\mathcal{B}\subset \mathbb{P}^{n+1} \times Gr(1,m+1)$), then : 
\begin{center}
    $c\geq a-n$ (resp, $c\geq b-m$).
\end{center}
\end{prop}
\begin{proof}
 First, as $\mathcal{B}$ is an irreducible subvariety of $ Gr(1,n+1)\times \mathbb{P}^{m+1} \sqcup \mathbb{P}^{n+1}\times Gr(1,m+1)$, the irreducibility of $\mathcal{B}$ implies that either $\mathcal{B} \subset Gr(1,n+1)\times \mathbb{P}^{m+1}$ or $\mathcal{B} \subset \mathbb{P}^{n+1}\times Gr(1,m+1)$. Suppose now that $\mathcal{B} \subset Gr(1,n+1)\times \mathbb{P}^{m+1}$ and that $c\leq a-n-1$. Then, for $t \in T, y\in \mathbb{P}^{1}$, the fibre $\phi^{-1}(t,y)$ lie on a line $l_{(t,y)}$ of the form $(l'_{(t,y)},z) \subset \mathbb{P}^{n+1} \times \mathbb{P}^{m+1}$, where $z\in \mathbb{P}^{m+1}$ and $l'_{(t,y)} \subset \mathbb{P}^{n+1}$ is a line. Moreover, it must satisfy the Cayley-Bacharach condition with respect to $\vert K_{X} \vert$. But $X_{z}$ is a hypersurface of degree $a$ in $\mathbb{P}^{n+1}$, hence $K_{X_{z}} = (a-n-2)H$ where $H$ is the hyperplane divisor of $\mathbb{P}^{n+1}$. Then, one can produce $E' \in \vert K_{X_{z}} \vert$ such that $E'$ passes through all the points in $pr_{1}(\phi^{-1}(t,y))$ but one as there are at most $a-n-1$ points in $pr_{1}(\phi^{-1}(t,y))$. Then, we can produce a divisor $E \in \vert K_{X} \vert \cong \vert \mathcal{O}(a-n-2,b-m-2)\vert $ such that $E$ passes through any point of $\phi^{-1}(t,y)$ but one. Indeed, consider a divisor $E'' \in \vert \mathcal{O}_{\mathbb{P}^{n+1}\times \mathbb{P}^{m+1}}(b-m-2) \vert $ passing through $w\in \mathbb{P}^{m+1}$. Then, the divisor $E'\boxtimes E''$ belongs to $\vert K_{X} \vert $ and passes through every points of $\phi^{-1}(t,y)$ but one. This implies that the fibre does not respect the Cayley-Bacharach condition with respect to $\vert K_{X} \vert$, which is a contradiction. Then $c\geq a-n$. The proof can be done the other way around to obtain that if $\mathcal{B} \subset \mathbb{P}^{n+1} \times Gr(1,m+1)$, then $c \geq b-m$.
\end{proof}

We now adapt \cite[Proposition 2.12]{bastianelli2019gonality} in our setting. In \cite{bastianelli2019gonality} this proposition enables the authors to show that the curves computing the covering gonality of a very general high degree hypersurface lie in specific cones. Following this statement, this leads to the next proposition: 

\begin{prop}
\label{LocalisationOfCurves}
Let $X$ be a very general and smooth hypersurface of bi-degree $(a,b)$ in $\mathbb{P}^{n+1} \times \mathbb{P}^{m+1}$, such that $a \geq 2n+m+1$ and $b\geq 2m+n+1$. Consider a covering family of c-gonal curve over $X$ with $c \leq k-3$. Then for each $t \in T$ there exists $x_{t} \in f(\mathcal{C}_{t})$ such that $f(\mathcal{C}_{t}) \subset V_{x_{t}}^{a-c}:=\{(\ell,z) \in Gr(1,n+1)\times\mathbb{P}^{m+1}, (z,\ell) \cdot X \geq (a-c)x_{t}\}$ or $f(C_{t}) \subset {W}^{b-c}_{x_{t}} : = \{\ell \in \mathbb{P}^{n+1}\times Gr(1,m+1), \ell \cdot X \geq (b-c)x_{t}\} $. Moreover, the gonality of the curve $f(\mathcal{C}_{t})$ is computed by projecting from the point $x_{t}$.
\end{prop}

\begin{proof}
First the assumptions about $a$ and $b$ ensure that $k\geq 0$, so that according to Proposition \ref{FibreCorres}, for general $t\in T, y \in \mathbb{P}^{1}$, the points $\{f(q_{1}),...,f(q_{c})\}$ lie on a line $\ell_{(y,t)} \in F(\mathbb{P}^{n+1}\times \mathbb{P}^{m+1})$. We now want to show that $\ell_{(t,y)}\cdot X-f(q_{1})-...-f(q_{c})$ does not vary with $y$ and is supported at a single point. This point will be the point $x_{t}$ of the statement.
Without loss of generality, as $\mathcal{B}$ is irreducible we can assume that the variety $\mathcal{B}$ lies in $Gr(1,n+1)\times \mathbb{P}^{m+1}$. The proof goes the same if $\mathcal{B}$ lies in $\mathbb{P}^{n+1} \times Gr(1,m+1)$.\\
Suppose that $\ell_{(t,y)}\cdot X-f(q_{1})-...-f(q_{c})$ does vary with $y$, this implies that $$\ell_{(t,y)} \cdot X-f(q_{1})-...-f(q_{c})$$ describes a curve $D_{t}$ in $\mathbb{P}^{n+1}\times \mathbb{P}^{m+1}$. But these curves $D_{t}$ cannot cover the hypersurface $X$. More precisely, the closure of the curves $D_{t}$ for $t\in T$ is not equal to $X$. Indeed, there is an irreducible component of $D_{t}$ having gonality strictly less than $\covgon(X)$. To estimate the gonality of an irreducible component of $D_{t}$, one remarks that $D_{t}$ is naturally dominated by the curve $F_{t}:=\{(x,y) \in D_{t}\times \mathbb{P}^{1} \vert x \in D_{t}, x \in \ell_{(t,y)} \}$. The projection of $F_{t}$ on $\mathbb{P}^{1}$ is of degree at most $a-c$ as there are at most $a-c$ points in the support of $\ell_{(t,y)} \cdot X -f(q_{1})-...-f(q_{c})$. Indeed, as $\mathcal{B} \subset Gr(1,n+1) \times \mathbb{P}^{m+1}$, $\ell_{(t,y)}$ intersect $X$ along one of the hypersurfaces of degree $a$ when $X$ is seen as family of hypersurfaces of degree $a$ over $\mathbb{P}^{n+1}$ through the projection $\pi_{2} : X \longrightarrow \mathbb{P}^{m+1}$. This means that the support of $\ell_{(t,y)} \cdot X$ contains at most $a$ point including $f(q_{1}),...,f(q_{c})$. Then every irreducible component of $D_{t}$ has at most gonality $a-c$. As $F_{t}$ dominates $D_{t}$, the same holds for any irreducible component of $D_{t}$. Using the refined estimate obtained in Proposition \ref{Estimate}, as $\mathcal{B} \subset Gr(1,n+1) \times \mathbb{P}^{m+1}$, we deduce that $c$ satisfies $c\geq a-n$, so that $a-c \leq a - a +n =n$. But, we have that $\covgon(X)\geq k+2 \geq n+m$, meaning the curves $D_{t}$ cannot cover $X$. \\
Now, let $S$ be an irreducible component of the closure of the curves $D_{t}$ and denote by $s$ its dimension. As the curves $D_{t}$ do not cover the hypersurface, it is clear that $1\leq s \leq n+m$. We get a contradiction by computing $\covgon(S)$ in two different ways. First, one can get more information on $S$ by looking at the ramification divisor $R$ of $p : \mathcal{I} \longrightarrow \widetilde{\mathcal{B}}$. For a very general $x\in S $, there is a $n+m+1-s$ dimensional family of line of $\mathcal{B}$, passing through $x$. Then, we have an irreducible component $Z$ of $R$ that dominates $S$ through $\mu : \mathcal{I} \longrightarrow \mathbb{P}^{n+1} \times \mathbb{P}^{m+1}$, moreover we choose $Z$ to dominate $\widetilde{\mathcal{B}}$ through $p$. In a more concrete way, this means we are choosing an open dense set of points in $S$ together with an open dense set of lines in $\widetilde{\mathcal{B}}$ such that each point in $S$ belongs to one of these lines in $S$. \\
Using this, we have an upper bound on $\covgon(S)$. To do so, we consider the restricted map $p\vert_{Z} : Z \longrightarrow \widetilde{\mathcal{B}}$ and denote $e$ its degree. We show that $\covgon(S) \leq e$. To do so, remark that $\widetilde{\mathcal{B}}$ is covered by rational curves that are obtained by making vary $y\in \mathbb{P}^{1}$ in $\ell_{(t,y)}$ for a value of $t\in T$ fixed. For $t \in T$ denote by $H_{t}$ the obtained curve. Its pullback under $p\vert_{Z}$ meets $Z$ along a curve $J_{t}$ which by construction dominates $D_{t}$. Then, as one recalls that $S$ is swept out by the closure of the curves $D_{t}$, we get that 
\[
\covgon(S)\leq \gonality(J_{t}) \leq \deg \left(p\vert_{H_{t}}\right)\leq e.
\]
We now give an estimate of $e$. By a classical computation, as in \cite[Lemma 1.1]{arrondo2002line}, we get that for a line $\ell \in \mathcal{B}$, $p^{-1}(\ell)$ meets $R$ at $n$ points counted with multiplicity. Indeed, this line $\ell$ belongs to $\mathcal{B}_{z}$ for a certain $z\in \mathbb{P}^{m+1}$. Then, applying the result \cite[Lemma 1.1]{arrondo2002line} yields the desired inequality.  Then, using \cite[Corollary A.6]{degreehypersurfaces}, we have that $\order_{Z}(R)\geq (n+m-1)-s$, which implies that $e(n+m+1-s)\leq n-1$, this yields $e\leq \frac{n-1}{n+m+1-s}$. \\
But, using Lemma \ref{StapUll} combined with the fact that $k \geq m+n$ (as $a\geq 2n+m+1$, $b\geq 2m+n+1$) gives the following control over $\covgon(S)$: 
\begin{center}
    $\covgon(S)\geq s\geq s+k-n-m$.
\end{center}
This implies that $s\leq \frac{n-1}{n+m+1-s} \leq 1+\frac{s}{n+m+1-s}$, which fails whenever $1\leq s \leq n+m+1$. This contradiction concludes the proof.
\end{proof}

\vspace{0.35cm}
This last proposition enables us to finally prove the desired lower bound on $\covgon(X)$. To do so, we now show that the curves on $X$ which compute the covering gonality must actually live in the fibres of the family of hypersurface $\pi_{1} : X \longrightarrow \mathbb{P}^{n+1}$ or $ \pi_{2} : X \longrightarrow \mathbb{P}^{m+1}$, depending on whether the covering gonality of a smooth and very general hypersurface of degree $a$ in $\mathbb{P}^{n+1}$ is greater than the covering gonality of a smooth and very general hypersurface of degree $b$ in $\mathbb{P}^{m+1}$.

\vspace{0.3cm}

We recall the result obtained in \cite{bastianelli2019gonality}: 

\begin{thm}[{\cite[Theorem 4.1]{bastianelli2019gonality}}]
\label{Bastianelli}
Let $X \subset \mathbb{P}^{n+1}$ be a very general irreducible smooth hypersurface of degree $d\geq 2n+2$. Then: 
\begin{center}
    $ d-\lfloor \frac{\sqrt{16n+9}-1}{2} \rfloor \leq \covgon(X)\leq d-\lfloor \frac{\sqrt{16n+1}-1}{2} \rfloor $.
\end{center}

where both sides of the inequality coincide whenever $n \in \{4\alpha^{2}+3\alpha, 4\alpha^{2}+5\alpha+1, \alpha \in \mathbb{N}\}$
Moreover, the upper bound is obtained by constructing a covering family of $d-\lfloor \frac{\sqrt{16n+1}-1}{2} \rfloor $-gonal curves on $X$. 
\end{thm}

Combining this theorem with the previous statement we get the proof of the first main result of this work that is : 

\begin{thm}

Let $X \subset \mathbb{P}^{n+1}\times \mathbb{P}^{m+1}$ with $m,n \geq 2$ be a very general irreducible smooth hypersurface of bi-degree $(a,b)$, such that $a \geq 2n+m+1,b\geq 2m+n+1$. Then: 
\begin{center}
    ${ \minimum \Big\{a-\lfloor \frac{\sqrt{16n+9}-1}{2} \rfloor,b-\lfloor \frac{\sqrt{16m+9}-1}{2} \rfloor \Big\} \leq \covgon(X)\leq \minimum \Big\{a-\lfloor \frac{\sqrt{16n+1}-1}{2} \rfloor,b-\lfloor \frac{\sqrt{16m+1}-1}{2} \rfloor \Big\} }$.
\end{center}
Moreover, both sides of the inequality coincide whenever $m,n \in \{4\alpha^{2}+3\alpha, 4\alpha^{2}+5\alpha+1, \alpha \in \mathbb{N}\}$
\end{thm}

\begin{proof}
As stated throughout this work, $X$ can alternatively be seen as a family of hypersurfaces of degree $a$ over $\mathbb{P}^{m+1}$ or as a family of hypersurfaces of degree $b$ over $\mathbb{P}^{n+1}$. Proposition \ref{LocalisationOfCurves} implies that $f(C_{t})$ is contained in $V_{x_{t}}^{a-c}$ for a certain $x_{t} \in X $ or in $ W^{b-c}_{x_{t}}$. Suppose, for instance that $f(C_{t}) \subset V_{x_{t}}^{a-c}$. More precisely, as a line that is in $V_{x_{t}}^{a-c}$ is in $Gr(1,n+1)\times \pi_{2}(x_{t})$, this shows that for a general $t\in T$ we have $f(\mathcal{C}_{t})\subset \pi^{-1}_{2}(\pi_{2}(x_{t}))$, i.e. the curve is on one of the hypersurfaces of the family. This means, that computing $\covgon(X)$ can be done fibrewise. Using Theorem \ref{Bastianelli}, this implies that $\covgon(X)\geq \minimum \Big\{a-\lfloor \frac{\sqrt{16n+9}-1}{2} \rfloor,b-\lfloor \frac{\sqrt{16m+9}-1}{2} \rfloor \Big\}$.
Moreover, the upper bound is clear as one can construct a family of curves on $X$ by constructing it fibrewise using the family of curves mentioned in Theorem \ref{Bastianelli}. This yields the desired upper bound. As in \ref{Bastianelli}, both side of the inquality coincide whenever Moreover, both sides of the inequality coincide whenever $m,n \in \{4\alpha^{2}+3\alpha, 4\alpha^{2}+5\alpha+1, \alpha \in \mathbb{N}\}$. 
\end{proof}
\begin{rmq}
We would like to make some comments on the assumption on the range of the bi-degree $(a,b)$ of the hypersurfaces $X$ of our theorem. Indeed, as already stated before, the assumptions $a\geq 2m+n+1, b\geq 2n+m+1$ are only required to get the lower bound on $\covgon(X)$; while the upper bound is achieved as soon as $a \geq 2n+2, b\geq 2m+2$. This was required to carry out the computations made during the proof of the Proposition \ref{LocalisationOfCurves}. Following \cite{yeong2022algebraic}, these numerical assumptions are actually the ones ensuring that $X$ is algebraically hyperbolic, otherwise, Yeong shows that $X$ contains a line. As the proof relied repeatedly on the fact that $X$ does not contain any line, one would need to change the method of the proof to possibly lower the bound on $a$ and $b$.
\end{rmq}

\section{Measure of association for hypersurfaces}
In this final section, we study joint measures of association for pairs of varieties. Lazarsfeld and Martin initiated this line of investigation in \cite{lazarsfeld2023measures}. Their motivation is to explore how birational complexity can be measured not only with respect to $\mathbb{P}^{n}$ but through correspondences between two given varieties. Just as there are various measures of rationality for individual varieties, one may also introduce different joint measures of association. In our case we are interested in the measure derived from the covering gonality that we will call the joint covering gonality. To introduce this invariant, let $X,Y \subset \mathbb{P}^{n+1} \times \mathbb{P}^{m+1}$ be smooth irreducible hypersurfaces of bi-degrees $(a_1,b_1)$ and $(a_2,b_2)$, respectively. Consider $Z$ as in the following diagram:

\begin{center}
\begin{tikzcd}
    & Z \subset X \times Y \arrow[dl, "p_{X}"'] \arrow[dr, "p_{Y}"] \\
    X & & Y
\end{tikzcd}
\end{center}
\vspace{0.5cm}

Here, $p_{X}$ and $p_{Y}$ are generically finite dominant morphisms so that $Z$ is a correspondence in $X\times Y$. A natural problem is to compute the minimal degree of irrationality of such a correspondence, and similarly, to estimate its minimal covering gonality. The larger these invariants are, the more birationally unrelated the varieties $X$ and $Y$ can be regarded. For the remainder of this section, we focus on the following notion of joint covering gonality :  
\begin{ddef}

The \emph{joint covering gonality} of the pair $(X,Y)$ is
\begin{center}

   $\operatorname{cov.gon}(X,Y) := \min_{Z \subset X\times Y} \operatorname{cov.gon}(Z)$
\end{center}
where the minimum is taken over all irreducible correspondences $Z \subset X\times Y$ with
generically finite dominant projections to both $X$ and $Y$.

\end{ddef}
Motivated by \cite[Theorem 3.4]{lazarsfeld2023measures}, we establish the following lower bound for the joint covering gonality: 

\begin{thm}
Let $X, Y \subset \mathbb{P}^{n+1} \times \mathbb{P}^{m+1}$ be very general smooth hypersurfaces 
of bi-degrees $(a_1,b_1)$ and $(a_2,b_2)$ respectively, with $a_{1},a_{2} \geq 2n+m+2$ and $b_{1},b_{2} \geq 2m+n+2$. Then : 
\[
\operatorname{cov.gon}(X,Y) \ge k_1 + k_2 - n - m
\]

where : 

$k_1 = \min\{a_1 - n - 3, b_1 - m - 3\}, \quad
k_2 = \min\{a_2 - n - 3, b_2 - m - 3\}$

\end{thm}

Our argument adapts the general strategy of Lazarsfeld–Martin in \cite[Theorem 3.4]{lazarsfeld2023measures}. This approach builds on the classical variational method of \cite{ein1988subvarieties} and \cite{voisin1996conjecture}. The idea is to show that the positivity of the canonical bundles of $X$ and 
$Y$ induces positivity on the canonical bundle of 
$Z$, thereby imposing geometric constraints that enable to control $\covgon(Z)$. \\
To make the idea precise, let us introduce the universal families of hypersurfaces in $\mathbb{P}^{n+1} \times \mathbb{P}^{m+1}$ of bi-degree $(a_{1},b_{1})$ and $(a_{2},b_{2})$. We will denote them $\mu :  \mathcal{X} \longrightarrow N$ and $\upsilon : \mathcal{Y} \longrightarrow M $, with $\dim(N)=p$ and $\dim(M)=q$. We apply \cite[Prop. 3.1]{lazarsfeld2023measures} to these families, which yields the following key statement. 

\begin{prop}
\label{PositivityZ}
Given a correspondence $Z \subset X \times Y$, then for every pair of integers $i,j $ with $i+j=n+m+1$, there exists a generically surjective morphism : 
\[
    p_{X}^{*}(\Omega^{p+i}_{\mathcal{X}}\vert_{X})\otimes p_{Y}^{*}(\Omega^{q+j}_{\mathcal{Y}}\vert_{Y}) \longrightarrow \Omega^{n+m+1}_{Z}
\]
\end{prop}

This observation suggests the following heuristic: if $X$ satisfies $(BVA)_{p}$ and $Y$ satisfies $(BVA)_{p'}$ for some $p,p' \in \mathbb{N}$, then one might expect that $Z$ satisfies $(BVA)_{p+p'}$. While this expectation does not hold in full generality, the next proposition establishes a weaker form of this phenomenon, which will suffice for proving the desired inequality.

\begin{prop}
\label{Inequality}
Consider $X,Y \subset \mathbb{P}^{n+1} \times \mathbb{P}^{m+1}$ two very general smooth hypersurfaces such that $X \in |\mathcal{O}(a_{1},b_{1})|$ and $Y \in |\mathcal{O}(a_{2},b_{2})|$. Let $Z$ be a correspondence between $X,Y$. Then, for any two integers $i,j \in \mathbb{N}$ such that $i+j=n+m+1$, we have the inequality:
\begin{equation*}
\begin{split}
K_{Z} \succeq & \min\{a_{1}-2n-m-3+i, b_{1}-2m-n-3+i\} H_{X} \\
& + \min\{a_{2}-2n-m-3+j, b_{2}-2m-n-3+j\} H_{Y}
\end{split}
\end{equation*}
where $H_{X}$ (resp. $H_{Y}$) is the pullback on $Z$ of $\mathcal{O}_{X}(1,1)$ (resp. $\mathcal{O}_{Y}(1,1)$) and for any two divisors $D,D'$, $D \succeq D'$ means that $D-D'$ is an effective divisor. 
\end{prop}

\begin{proof}
As defined above, let  $\mathcal{X}\longrightarrow M$, $\mathcal{Y} \longrightarrow N$ be the two universal families of hypersurfaces with $\dim(N)=p$ and $\dim(M)=q$. As stated in \cite[Lemma 3.20]{stapleton2017degree}, the restricted tangent bundles $T_{\mathcal{X}}\vert_{X}(1,1)$ and $T_{\mathcal{Y}}\vert_{Y}(1,1)$ are globally generated. Using that, we decompose the vector bundles appearing in the last proposition to derive the desired relation. Indeed, we have the following : 
\begin{align*}
\Omega_{\mathcal{X}}^{p+i}|_{X} &= \left( \bigwedge^{(n+m+1+p)-(p+i)} T_{\mathcal{X}}|_{X} \right) \otimes K_{X} \\
&= \left( \bigwedge^{(n+m+1)-i} T_{\mathcal{X}}|_{X} \right) \otimes \mathcal{O}_{X}(a_1-n-2, b_1-m-2) \\
&= \left( \bigwedge^{(n+m+1)-i} (T_{\mathcal{X}}|_{X} \otimes \mathcal{O}_{X}(1,1)) \right) \otimes \mathcal{O}_{X}(a_1+i-2n-m-3, b_1+i-2m-n-3)
\end{align*}

Moreover, we get by the same method:

\[
\Omega_{\mathcal{Y}}^{q+j}|_{Y} = \left( \bigwedge^{(n+m+1)-j} (T_{\mathcal{Y}}|_{Y} \otimes \mathcal{O}_{Y}(1,1)) \right) \otimes \mathcal{O}_{Y}(a_2+j-2n-m-3, b_2+j-2m-n-3)
\]

As $T_{\mathcal{Y}}|_{Y} \otimes \mathcal{O}_{Y}(1,1)$ and $T_{\mathcal{X}}|_{X} \otimes \mathcal{O}_{X}(1,1)$ are globally generated, Proposition \ref{PositivityZ} yields the desired inequality.
\end{proof}

With that, we can now finally prove the previously announced lower bound on $\covgon(X,Y)$

\begin{proof}
Consider a covering family of c-gonal curves denoted $ \pi: \mathcal{C} \longrightarrow \mathcal{T}$, $f: \mathcal{C} \longrightarrow Z$. Then, we have that $K_{\mathcal{C}} \succeq K_{Z}$ because of the adjunction formula. Moreover, we have that if $C$ is a general fiber of the family $\pi : \mathcal{C} \longrightarrow \mathcal{T}$, $K_{C}=K_{\mathcal{C}}\vert_{C}$. Fix the pullback of the hyperplane divisor $H_{X}$ (resp. $H_{Y}$) in $X$ (resp) $Y$, Proposition \ref{Inequality} yields : 
\begin{equation*}
\begin{aligned}
K_{C} &\succeq \min\{a_{1}-2n-m-3+i, b_{1}-2m-n-3+i\} H_{X} \\
& \quad + \min\{a_{2}-2n-m-3+j, b_{2}-2m-n-3+j\} H_{Y}
\end{aligned}
\end{equation*} 
\vspace{0.5cm}
Since, a priori, $C$ does not map birationally to $X$ or $Y$, there is also no reason to expect $H_X,H_Y$ to be very ample. Nevertheless, one can circumvent this difficulty by adapting the strategy adopted by  Lazarsfeld and Martin in \cite[Lemma 3.6]{lazarsfeld2023measures}. We briefly recall it.
To do so consider, denote $D_{1}$ (resp $D_{2}$) the normalisation of the image of the curve $C$ in $X$ (resp. $Y$). Then, $D_{1}$ (resp. $D_{2}$) map birationally into $X$ (resp. Y) and $C$ maps birationally its image in $D_{1} \times D_{2}$. One can remark that, $\gonality(D_{1}) \geq \min\{a_{1}-n,b_{1}-m\}$ and $\gonality(D_{2}) \geq \min \{a_{2}-n,b_{2}-m\}$ as $D_{1}$ lies on $X$ and $D_{2}$ lies on $Y$.  
Then, we can build line bundles that separates the points on $D_{1}$ and $D_{2}$, which behave well under the pullback on $C$:  
\begin{center}
    $L_{1}:=\mathcal{O}_{D{1}}(a_{1}-2n-m-3+i, b_{1}-2m-n-3+i)$, $L_{2}:=\mathcal{O}_{D_{2}}(a_{2}-2n-m-3+j, b_{2}-2m-n-3+j)$
\end{center} 
\vspace{0.3cm}

Then, consider $\phi : C \longrightarrow \mathbb{P}^{1}$ of degree $d\leq  k_{1}+k_{2}-n-m-1$, \cite[Lemma 3.6]{lazarsfeld2023measures} imposes that $L_{1}\boxtimes L_{2}$ separates the points in the fibre of $\phi$. This means that $K_{C}$ shares the same property yielding a contradiction. Thus the desired inequality follows. 
\end{proof}

\begin{rmq}
It is natural to ask whether sharper bounds hold. For hypersurfaces in projective space, even the joint covering gonality is not fully understood. For example, if $X_d, Y_{d'} \subset \mathbb{P}^{n+1}$ are hypersurfaces of degrees $d,d'$, one expects $\mathrm{\covgon}(X_d,Y_{d'})$ to grow multiplicatively with $d$ and $d'$. By analogy, in our setting one might conjecture $\operatorname{cov.gon}(X,Y) \sim \min\{a_{1},b_{1}\}\cdot \min\{a_{2},b_{2}\}$ when all degrees are large enough. 
\end{rmq}

\bigskip

\noindent
\textsc{Universit\'e de Lorraine, Institut \'Elie Cartan de Lorraine, Nancy, France}

\texttt{E-mail : raphael.hiault@univ-lorraine.fr}
\bibliographystyle{alpha}
\bibliography{main}
\end{document}